\documentclass[11pt]{amsart}
\usepackage[utf8]{inputenc} 
\usepackage{mathtools,amssymb,amsxtra,mathrsfs}
\usepackage{bm} 
\usepackage{xcolor,enumitem}
\usepackage{ragged2e}
\usepackage{fancyhdr}

\usepackage{indentfirst}

\usepackage{graphicx} 

\newcommand\rsmraise[1]{%
  \ifx#1\displaystyle 1 \else
    \ifx#1\textstyle .3 \else
      \ifx#1\scriptstyle .2 \else
        .1%
      \fi
    \fi
  \fi}

\theoremstyle{plain}      

\theoremstyle{definition} 
\newtheorem{thrm}{Theorem}[section]
\newtheorem*{thrm*}{Theorem}
\newtheorem{cor}[thrm]{Corollary}
\newtheorem{lemma}[thrm]{Lemma}
\newtheorem*{lemma*}{Lemma}

\newtheorem*{propn*}{Proposition}
\newtheorem{defn}{Definition}[section]
\newtheorem*{defn*}{Definition}
\newtheorem*{rmk*}{Remark}

\newtheorem{que}{Question}
\newtheorem{cla}{Claim}

\theoremstyle{remark}     

\newtheoremstyle{cited}
  {3pt}
  {3pt}
  {}
  {}
  {\bfseries}
  {.}
  {.5em}
  {\thmname{#1} \thmnumber{#2} \thmnote{\normalfont#3}}

\theoremstyle{cited}

\usepackage{lipsum}

\newcommand\blfootnote[1]{%
  \begingroup
  \renewcommand\thefootnote{}\footnote{#1}%
  \addtocounter{footnote}{-1}%
  \endgroup
}

\begin{document}

\title{Turing Invariant Sets and the Perfect Set Property}
\author{Clovis Hamel, Haim Horowitz and Saharon Shelah}



\blfootnote{Publication 1180 of the third author}

\maketitle
\begin{abstract}
    We show that $ZF+DC+$"all Turing invariant sets of reals have the perfect set property" implies that all sets of reals have the perfect set property. We also show that this result generalizes to all countable analytic equivalence relations.
    
\end{abstract}

\section{Introduction}

The results of this paper are motivated by the following well known open problem:

\begin{que}
Does Turing determinacy ($TD$) imply $AD$?
\end{que}

It was proven by Woodin that $TD+DC_{\mathbb R}+V=L(\mathbb R)$ implies $AD$, however, the above question remained open. Inspired by this question, we asked the following analogous question:

\begin{que}
Let $\Gamma$ be a regularity property (e.g., the perfect set property, Lebesgue measurability, etc), does $ZF+DC+$"all Turing invariant sets have property $\Gamma$" imply that all sets of reals have property $\Gamma$?
\end{que}

The main result of this paper answers the above question in the affirmative when $\Gamma$ is the perfect set property. We also observe that Turing equivalence can be replaced by a more general collection of countable Borel equivalence relations. Furthermore, we provide a recursion theoretic argument due to Liang Yu, showing that the result generalizes to all countable analytic equivalence relations.

$\\$
The paper is organized as follows: The main result of Section 2 is an affirmative answer to Question 2 where $\Gamma$ is the perfect set property. In Section 3, we present Yu's argument generalizing the result of Section 2 to all countable analytic equivalence relations. In Section 4, we present several open problems.
$\\$

$\textbf{Acknowledgment}$: We thank Liang Yu for allowing us to include his arguments in Section 3.

\section{The main result}
\begin{rmk*} 1. In what follows, given a Turing machine $M$ and a real $\eta$, we write $M(\eta)$ for the real computed from $\eta$ by $M$ (i.e. via the associated Turing functional).

2. Although the following proof is using  $DC$ and the recursion theoretic proof in the next section uses $AC_{\omega}$, both can be eliminated by observing that if $V\models ZF+$"all Turing invariant sets have the perfect set property" and $X\in V$ is a set of reals, then $HOD(\mathbb R, X) \models ZF+DC+$"all Turing invariant sets have the perfect set property".
\end{rmk*}

\begin{thrm}
($ZF+DC$): The perfect set property for all Turing invariant sets of reals implies the perfect set property for all sets of reals.

\end{thrm} 
\begin{proof}
Let $\kappa$ be large enough and fix a countable elementary submodel $N$ of $(H(\kappa), \in)$. Fix a perfect tree $T$ such that, for every $n<\omega$ and pairwise distinct $\eta_0,...,\eta_{n-1} \in lim(T)$, $(\eta_0,...,\eta_{n-1})$ is $N$-generic for $(2^{<\omega}, \leq)^n$. Clearly, if $\eta \neq \nu \in lim(T)$, then $\eta$ and $\nu$ are not Turing equivalent. We break the proof into five claims:

\begin{cla}
For every pair $(M_1, M_2)$ of Turing machines, there is a natural number $n=n(M_1, M_2)$ such that, for every $\eta \neq \nu \in lim(T)$, if $\eta \restriction n=\nu \restriction n $ then: 
\begin{itemize}
    \item[(a)] $M_2(M_1(\eta))=\eta$ iff $M_2(M_1(\nu))=\nu$.
    \item[(b)] If $M_1(\eta)=M_1(\nu)$, then $M_1(\eta) \in N$.
    \item[(c)] $M_1(\eta)=M_2(\eta)$ iff $M_1(\nu)=M_2(\nu)$.
    \item[(d)] If $M_2(M_1(\eta))=\eta$, then $M_1(\eta)$ is not Turing equivalent to $M_1(\nu)$.
\end{itemize}
\end{cla}

\begin{proof}[Proof of Claim 1]
First we note that clause (d) follows from clause (a): Suppose towards contradiction that $M_2(M_1(\eta))=\eta$ and $M_1(\eta)$ is Turing equivalent to $M_1(\nu)$, then $\eta=M_2(M_1(\eta))$ is Turing equivalent to $M_2(M_1(\nu))$. By clause (a), $M_2(M_1(\nu))=\nu$, hence $\eta$ is Turing equivalent to $\nu$, a contradiction. Clause (b) follows from the mutual genericity over $N$ of the branches in $lim(T)$. We shall now prove clause (a), the proof of clause (c) is similar. Given $\eta \in lim(T)$, there is some $n$ such that $\eta \restriction n$ (as a Cohen condition) decides the truth value of $"M_2(M_1(\eta))=\eta"$ and such that for every $\eta \restriction n \leq \nu \in lim(T)$, $M_2(M_1(\nu))=\nu$ iff $M_2(M_1(\eta))=\eta$. Denote $\eta \restriction n$ by $c_{\eta}$ and the set of $\nu \in lim(T)$ such that $\nu \restriction n= c_{\eta}$ by $U_{\eta}$. By compactness, there is some $k<\omega$ and $\eta_0,...,\eta_{k-1}$ such that $2^{\omega}=U_{\eta_0} \cup ... \cup U_{\eta_{k-1}}$. Let $n=n(M_1, M_2)$ be the maximum length of $\{c_{\eta_i} : i<k \}$, then $n$ is as required. This completes the proof of Claim 1.
\end{proof}

Now let $((M_{n,0}, M_{n,1}) : n<\omega)$ be an enumeration of all ordered pairs of Turing machines, where $M_{0, 0}$ and $M_{0, 1}$ act as the identity function. For $n<\omega$, let $X_n$ be the set of all $\eta \in lim(T)$ such that:

\begin{enumerate}
    \item $M_{n,1}(M_{n,0}(\eta))=\eta$, and
    \item $M_{n,0}(\eta) \notin \{M_{l,0}(\eta) : l<n, M_{l,1}(M_{l,0}(\eta))=\eta\}$.
    
\end{enumerate}

Now, for each $n<\omega$, let $Y_n=\{M_{n,0}(\eta) : \eta \in X_n\}$.

\begin{cla}
For every $n<\omega$, there exists $k_n$ such that, for every $\eta \in lim(T)$, $\eta \restriction k_n$ determines the truth value of $"\eta \in X_n"$. It follows that each $X_n$ is closed, and hence, each $Y_n$ is closed (being a continuous image of a compact set).
\end{cla}
\begin{proof}[Proof of Claim 2]
This is similar to the proof of Claim 1. Given $\eta \in lim(T)$, there is some $n<\omega$ such that $\eta \restriction n$ decides the membership of the Cohen generic in $X_n$. Denote $\eta \restriction n$ by $c_{\eta}$ and denote the set of all $\nu \in lim(T)$ such that $\nu \restriction n=c_{\eta}$ by $U_{\eta}$. Again, by compactness, there is some $k<\omega$ and $\eta_0,...,\eta_{k-1}$ such that $2^{\omega}=U_{\eta_0} \cup ... \cup U_{\eta_{k-1}}$, and we let $k_n$ be the maximum length of $\{c_{\eta_i} : i<k \}$. This completes the proof of Claim 2.
\end{proof}

\begin{cla}
$\underset{n<\omega}{\bigcup}Y_n$ is the closure of $lim(T)$ under Turing equivalence.
\end{cla}
\begin{proof}[Proof of Claim 3]
Every element of $\underset{n<\omega}{\bigcup}Y_n$ is Turing equivalent to an element of $lim(T)$, by the definition of $X_n$ and $Y_n$. Suppose that $\nu$ is Turing equivalent to some $\eta \in lim(T)$, so there are Turing machines $M_0$ and $M_1$ such that $\nu=M_0(\eta)$ and $\eta=M_1(\nu)$. There is some $n<\omega$ such that $(M_0,M_1)=(M_{n,0},M_{n,1})$, and therefore, $M_{n,1}(M_{n,0}(\eta))=M_1(M_0(\eta))=\eta$ and $M_{n,0}(\eta)=\nu$. Let $m<\omega$ be the minimal natural number with this property, then $\eta \in X_m$ and $\nu=M_{m,0}(\eta) \in Y_m$. This completes the proof of Claim 3.
\end{proof}

\begin{cla}
$\{Y_n : n<\omega \}$ is a family of pairwise disjoint sets.
\end{cla}

\begin{proof}[Proof of Claim 4]
Suppose towards contradiction that there is some $\eta \in Y_n \cap Y_m$ where $m<n$. Let $\nu \in X_n$ and $\nu' \in X_m$ such that $\eta=M_{n,0}(\nu)=M_{m,0}(\nu')$, then $\nu'=M_{m,1}(M_{m,0}(\nu'))=M_{m,1}(M_{n,0}(\nu)) \in N[\nu]$. By the mutual genericity of $\nu$ and $\nu'$, it must be the case that $\nu=\nu'$, so $\eta=M_{n,0}(\nu)=M_{m,0}(\nu)$. But this contradicts the fact that $\nu \in X_n$. It follows that $Y_n \cap Y_m=\emptyset$, which completes the proof of the Claim 4.
\end{proof}
Fix a homeomorphism $F: 2^{\omega} \rightarrow lim(T)$. In order to show that every uncountable $A \subseteq 2^{\omega}$ contains a perfect subset, it suffices to show that every uncountable $B\subseteq lim(T)$ contains a perfect subset: If $A\subseteq 2^{\omega}$ is uncountable, then $B=\{F(\eta) : \eta \in A\} \subseteq lim(T)$ is uncountable and contains a perfect subset, and by $F$ being a homeomorphism, so does $A=F^{-1}(B)$.
\\
\\
Now let $A \subseteq lim(T)$ be uncountable, we shall find a perfect subset of $A$. For $n<\omega$, let $A_{1, n}=\{M_{n,0}(\eta) : \eta \in A \cap X_n\}$ and let $A_2=\underset{n<\omega}{\bigcup}A_{1,n}$.

\begin{cla}
$A_2$ is Turing invariant.
\end{cla}

\begin{proof}[Proof of Claim 5]
We shall prove that $A_2$ is the closure of $A$ under Turing equivalence. Obviously, every element of $A_2$ is Turing equivalent to an element of $A$, by the definition of $A_2$. Suppose now that $\nu$ is Turing equivalent to some $\eta \in A$, then there is a minimal $n<\omega$ such that $M_{n,1}(M_{n,0}(\eta))=\eta$ and $M_{n,0}(\eta)=\nu$. Therefore, $\eta \in X_n \cap A$, hence $\nu \in A_{1,n} \subseteq A_2$. This completes the proof Claim 5.
\end{proof}

As $A_2$ is Turing invariant and uncountable (recalling that it contains $A$), by the assumption, it contains a perfect subset $P$. Note that $A_{1,n} \subseteq Y_n$ for every $n<\omega$, so $P\subseteq \underset{n<\omega}{\bigcup}Y_n$. As the $Y_n$ are closed and pairwise disjoint, we may assume WLOG that there is some $n^*<\omega$ such that $P\subseteq Y_{n^*}$, so $P\subseteq A_2 \cap Y_{n^*}=A_{1,n^*}$ (recalling that the $Y_n$ are pairwise disjoint and $A_{1, n} \subseteq Y_n$). Let $A_3=\{M_{n^*,1}(\eta) : \eta \in P\}$, then $A_3 \subseteq A$. Therefore, it suffices to show that $A_3$ is perfect. Note that if $\eta, \eta' \in P \subseteq Y_{n^*}$, then there are $\nu, \nu' \in X_{n^*}$ such that $\eta=M_{n^*,0}(\nu)$ and $\eta'=M_{n^*,0}(\nu')$, and therefore, if $M_{n^*,1}(\eta)=M_{n^*,1}(\eta')$, then $\nu=\nu'$ and $\eta=\eta'$, so $M_{n^*,1}$ is injective and continuous on $P$. Similarly, $M_{n^*,0}$ is injective on $A_3$: Note that if $\eta \in P \subseteq Y_{n^*}$, then $\eta=M_{n^*,0}(\nu)$ for some $\nu \in X_{n^*}$, hence $M_{n^*,1}(\eta)=\nu \in X_{n^*}$. Therefore, $A_3 \subseteq X_{n^*}$. Note that $M_{n^*,0}$ is injective on $X_{n^*}$, hence it follows that $M_{n^*,0}$ is injective and continuous on $A_3$. It's easy to verify that $M_{n^*,1}$ restricted to $P$ is the inverse of $M_{n^*,0}$ restricted to $A_3$, and it follows that $A_3$ is perfect. This completes the proof of Theorem 1.
\end{proof}

Finally, we observe that the above results can be generalized as follows:

\begin{defn}
Let $\mathcal F=\{f_n : n<\omega\}$ a countable family of ground model-definable partial continuous functions from a Polish space $X$ to itself and let $\{(g_{n,0}, g_{n,1}) : n<\omega\}$ be a fixed enumeration of all ordered pairs from $\mathcal F$. Let $E_{\mathcal F}$ be the following relation on $X$: $xE_{\mathcal F}y$ iff there is some $n<\omega$ such that $g_{n,0}(x)=y$ and $g_{n,1}(y)=x$. It's not hard to see that $E_{\mathcal F}$ is countable Borel equivalence relation on $X$.
\end{defn}

Note that the only property of the Turing equivalence relation that we used in our proof is that it has the form $E_{\mathcal F}$ where $\mathcal F$ is the collection of all functions of the form $M(\eta)=\nu$ where $M$ is a Turing machine. Therefore, we obtain the following corollary:

\begin{cor}  Assume $ZF+DC$. Let $E$ be a countable Borel equivalence relation of the form $E_{\mathcal F}$ where $\mathcal F$ is as above. If all $E$-invariant sets of reals have the perfect set property, then all sets of reals have the perfect set property. In particular, the above result holds for $E=E_0$.

\end{cor}

\section{A recursion theoretic proof and a generalization to all countable analytic equivalence relations}

In this section, we sketch a recursion theoretic reformulation of the proof of the main result from Section 2, and show how to modify the argument in order to generalize the result of Section 2 to all countable analytic equivalence relations. All arguments in this section are due to Liang Yu.
$\\$

We begin by sketching a recursion theoretic proof of Theorem 2.1:

\begin{proof}
Let $P\subset 2^{\omega}$ be a perfect set so that
\begin{itemize}
\item[(1)] any two different reals from $P$ have different degrees;
\item[(2)] any real in $P$ is hyperimmune-free.
\end{itemize}
 Now fix an uncountable set $A$ of reals. WLOG, we may assume that $A\subseteq P$. By Property (2), the Turing closure $[A]_T=\{x\mid \exists y \in A(\equiv_T x)\}=[A]_{tt}=\{x\mid \exists y \in A(\equiv_{tt} x)\}$, where $tt$ means truth-table reduction,  has a perfect subset $Q$.  Now for any pair of indexes of truth-table reductions $e,i \in \omega$, let $Q_{e,i}=\{z\mid \exists x\in Q(z=\Phi_e^x \wedge x=\Phi_i^z)\}$. Then for each $e,i$, $Q_{e,i}$ is a closed set and $\bigcup_{e,i}Q_{e,i}=[Q]_{tt}=[Q]_T$. Also by Property (1), $Q_{e,i}\cap P\subseteq A$ and $\bigcup_{e,i}(Q_{e,i}\cap P)=A\cap [Q]_T$ and $[Q]_T= [[Q]_T\cap A]_T$. By $DC$ (in fact $AC_{\omega}$), $[Q]_T\cap A$ is uncountable. By $DC$ ($AC_{\omega}$, again), there must be some $e,i$ so that $Q_{e,i}\cap A$ is uncountable. Then $Q_{e,i}\cap P=Q_{e,i}\cap A$ is an uncountable closed set and so must contain a perfect subset.
\end{proof}

\begin{thrm}In Theorem 2.1, Turing equivalence can be replaced by any countable analytic equivalence relation.

\end{thrm}

\begin{proof}

We begin with a short lemma: 

\begin{lemma}
If $x$ is $\Delta^1_1$-dominated, then for any $\alpha<\omega_1^{CK}$, there is some $\beta<\omega_1^{CK}$ so that $x^{(\alpha)}\leq x\oplus \emptyset^{(\beta)}$.
\end{lemma}
\begin{proof}
$x^{(\alpha)}$ is Turing equivalent to a $\Pi^0_1(x)$-singleton $f\in \omega^{\omega}$. Since  $x$ is $\Delta^1_1$-dominated, there is a hyperarithmetic function $g$ majorizing $f$. Then $f\leq_T x\oplus g\leq_T x\oplus \emptyset^{(\beta)}$ for some $\beta<\omega_1^{CK}$.
\end{proof}

We shall now return to the proof of Theorem 3.1. Suppose that $E$ is a countable analytic equivalence relation. We may assume that $E$ is a (lightface) $\Sigma^1_1$ countable equivalence relation. For the boldface case, we just need a relativization. By the property of $E$, for any pair $x,y$, $xEy$ implies $x\equiv_h y$, where $\leq_h$ is hyperarithmetic reduction (this follows from the fact that if a $\Sigma^1_1$ set is countable, then all of its members are hyperarithmetic, see e.g. Lemma 2.5.4 in [CY]). let $P\subset 2^{\omega}$ be a perfect set so that
\begin{itemize}
\item[(1')] any two different reals from $P$ have different hyperdegrees;
\item[(2')] any real in $P$ is $\Delta^1_1$-dominated (i.e. for any $x\in P$ and $f\in \omega^{\omega}$ with $f\leq_h x$, there is hyperarithmetic function $g$ dominating $f$. Note that this implies $\omega_1^x=\omega_1^{CK}$).
\end{itemize}
Now fix  an uncountable set $A$ of reals. WLOG, we may assume that $A\subseteq P$. Then by replacing the conditions (1), (2) and Turing reduction with (1'), (2') and hyperarithmetic reduction respectively, we may apply the same arguments as in the recursion theoretic proof of Theorem 2.1  together with Lemma 3.2 to prove that $A$ has a perfect subset.

\end{proof}

\section{open problems} 
As noted in the introduction, it is not known whether Turing determinacy implies $AD$. Furthermore, it’s not even known whether Turing determinacy implies weak consequences of $AD$ such as ”all sets of reals have property $\Gamma$” for a regularity property $\Gamma$. We therefore ask:

\begin{que} Let $\Gamma$ be a regularity property, does Turing determinacy imply that all sets of reals have property $\Gamma$?

\end{que}

\begin{que}
Does Turing determinacy imply that all Turing invariant sets of reals have the perfect set property? A positive answer to this question, combined with the results of this paper, will establish that Turing deter- minacy implies the perfect set property for all sets of reals, answering a question from [Sa] (see also [Lo]).
\end{que}

\begin{que}
For which countable Borel equivalence relations $E$ do we have that "all $E$-invariant sets are determined" imply $AD$? We note that recent progress on this problem has been made in [CFJ].
\end{que}

\begin{que}
For which Borel equivalence relations $E$ and regularity properties $\Gamma$ do we have that $ZF+DC+$"all $E$-invariant sets of reals have property $\Gamma$" imply "all sets of reals have property $\Gamma$"?
\end{que}

\section{References}

[CY] Chi Tat Chong, Liang Yu: Recursion Theory - Computational Aspects of Definability. de Gruyter series in logic and its applications 8, de Gruyter Oldenbourg 2015, ISBN 978-3-11-027555-1, pp. I-XIII, 1-306

[CFJ] Logan Crone, Lior Fishman and Stephen Jackson, Equivalence Relations and Determinacy, arXiv:2003.02238

[Lo] Benedikt Löwe, Turing Cones and Set Theory of the Reals, Archive for Mathematical Logic 40 (2001), pp. 651-664

[Sa] Ramez L. Sami, Turing determinacy and the continuum hypothesis, Archive for Mathematical Logic, October 1989, Volume 28, Issue 3, pp 149-154
\\
\\
(Clovis Hamel) E-mail address: chamel@math.utoronto.ca
\\
(Haim Horowitz) E-mail address: haim@math.toronto.edu
\\
\\
Department of Mathematics University of Toronto
\\
Bahen Centre, 40 St. George St., Room 6290 Toronto,
\\
Ontario, Canada M5S 2E4
\\
\\
(Saharon Shelah) E-mail address: shelah@math.huji.ac.il
\\
\\
Einstein Institute of Mathematics
\\
Edmond J. Safra Campus,
\\
The Hebrew University of Jerusalem.
\\
Givat Ram, Jerusalem 91904, Israel.
\\
\\
Department of Mathematics
\\
Hill Center - Busch Campus,
\\
Rutgers, The State University of New Jersey.
\\
110 Frelinghuysen Road, Piscataway, NJ 08854-8019 USA
\end{document}